\documentclass{amsart}
\usepackage{mathrsfs}
\usepackage{mathrsfs}
\usepackage{amsfonts}
\usepackage{cases}
\usepackage{latexsym}
\usepackage{amsmath}
\usepackage[all]{xy}
\usepackage{stmaryrd}
\usepackage{amsfonts}
\usepackage{amsmath,amssymb,amscd,bbm,amsthm,mathrsfs,dsfont}

\usepackage{fancyhdr}
\usepackage{amsxtra,ifthen}
\usepackage{verbatim}

\numberwithin{equation}{section}

\theoremstyle{plain}
\newtheorem{theorem}{Theorem}[section]
\newtheorem{lemma}[theorem]{Lemma}
\newtheorem{proposition}[theorem]{Proposition}
\newtheorem{corollary}[theorem]{Corollary}

\theoremstyle{definition}
\newtheorem{definition}[theorem]{Definition}
\newtheorem{example}[theorem]{Example}

\newtheorem{remark}[theorem]{Remark}

\begin{document}

\title[Lie Derivations of Dual Extension Algebras]
{Lie Derivations of Dual Extensions}

\author{Yanbo Li and Feng Wei}

\address{Li: School of Mathematics and Statistics, Northeastern
University at Qinhuangdao, Qinhuangdao, 066004, P.R. China}

\email{liyanbo707@163.com}

\address{Wei: School of Mathematics, Beijing Institute of
Technology, Beijing, 100081, P. R. China}

\email{daoshuo@hotmail.com}

\begin{abstract}
Let $K$ be a field and $\Gamma$ a finite quiver without oriented
cycles. Let $\Lambda$ be the path algebra $K(\Gamma, \rho)$ and
let $\mathscr{D}(\Lambda)$ be the dual extension of $\Lambda$. In
this paper, we prove that each Lie derivation of
$\mathscr{D}(\Lambda)$ is of the standard form.
\end{abstract}

\subjclass[2000]{15A78, 16W25, 47C05}

\keywords{Dual extension, generalized matrix algebra, Lie
derivation, }

\thanks{The first author of this work is supported by Fundamental Research
Funds for the Central Universities (N110423007).}

\maketitle

\section{Introduction}\label{xxsec1}

In the study of the representation theory of quasi-hereditary
algebras, Xi \cite{Xi1} defined dual extensions of algebras without
oriented cycles. Roughly speaking, these algebras $A$ are
constructed by adding to the ordinary quiver (without oriented
cycles) of a given algebra $B$ a reverse arrow for any original
arrow and extending the relations in a suitable way to this extended
quiver. They are a class of finite dimensional quasi-hereditary
algebras and they were detailedly investigated in \cite{DengXi1,
DengXi3, Xi2} by Deng and Xi. A dual extension algebra is a
BGG-algebra in the sense of R. Irving \cite{Irving}, that is, a
quasi-hereditary algebra with a duality which fixes all simple
modules. A much common more general construction, the twisted
doubles, were studied in \cite{DengXi2, KoenigXi, Xi3} by Deng,
Koenig and Xi.

Derivations and Lie derivations of associative algebras, as
classical linear mappings, play significant roles in various
mathematical areas, such as in Lie theory, matrix theory,
noncommutative algebras and operator algebras. Let $\mathcal{R}$ be
a commutative ring with identity, $\mathcal{A}$ be a unital algebra
over $\mathcal{R}$ and $\mathcal{Z(A)}$ be the center of
$\mathcal{A}$. We write $[a, b]=ab-ba$ for all $a,b\in \mathcal{A}$.
Let $\Theta\colon \mathcal{A}\longrightarrow \mathcal{A}$ be a
linear mapping. We call $\Theta$ an \textit{{\rm (}associative{\rm
)} derivation} if
$$
\Theta(ab)=\Theta(a)b+a\Theta(b)
$$
for all $a, b\in \mathcal{A}$. Further, $\Theta$ is called a
\textit{Lie derivation} if
$$
\Theta([a, b])=[\Theta(a), b]+[a, \Theta(b)]
$$
for all $a, b\in \mathcal{A}$. It is clear that every associative
derivation is a Lie derivation. But, the converse statement is not
true in general. Moreover, if $D\colon \mathcal{A}\longrightarrow
\mathcal{A}$ is an associative derivation and $\Delta\colon
\mathcal{A}\longrightarrow \mathcal{Z(A)}$ is a linear mapping such
that $\Delta([a, b])=0$ for all $a, b\in \mathcal{A}$, then the
mapping
$$
\Theta=D+\Delta, \eqno(\spadesuit)
$$ is a Lie derivation. We shall say that a
Lie derivation is \textit{standard} in the case where it can be
expressed in the preceding form. We call $\Theta$ a \textit{Jordan
derivation} if
$$
\Theta(a\circ b)=\Theta(a)\circ b+a\circ \Theta(b)
$$
for all $a, b\in \mathcal{A}$. Of course every associative
derivation is a Jordan derivation, while the converse statement is
not always true. We shall say that a Jordan derivation is
\textit{standard} if it is an associative derivation.

A common and popular problem in the study of Lie derivations is
whether they have the above mentioned standard form. Equivalently,
how every Lie derivation is approximate to a derivation to the
utmost extent. The first result in this aspect is due to Martindale,
who proved that each Lie derivation of a prime ring satisfying some
conditions is of the standard form in \cite{Martindale}. Alaminos et
al \cite{AlaminosBresarVillena} showed that every Lie derivation on
the full matrix algebra over a field of characteristic zero has the
standard form. Cheung \cite{Cheung} considered Lie derivations of
triangular algebras and gave a sufficient and necessary condition
which enables every Lie derivation to be standard. Benkovic studied
the structure of Lie derivations from a triangular algebra into its
bi-module in \cite{Benkovic}. The description of standard form on
Lie triple derivations of triangular algebras were obtained by Xiao
and Wei in \cite{XiaoWei2}. Recently, the current authors and Xiao
investigated the associative-type, Lie-type and Jordan-type linear
mappings of generalized matrix algebras. For details, we refer the
reader to \cite{LiWei1, LiWei2, LiWykWei, XiaoWei1}.

Path algebras of quivers naturally appear in the study of tensor
algebras of bimodules over semisimple algebras. It is well known
that any finite dimensional basic $K$-algebra is given by a quiver
with relations when $K$ is an algebraically closed field. In
\cite{GuoLi}, Guo and Li studied the Lie algebra of differential
operators on a path algebra $K\Gamma$ and related this Lie algebra
to the algebraic and combinatorial properties of the path algebra
$K\Gamma$. In \cite{LiWei2}, the current authors studied Lie
derivations of a class of path algebras of quivers without oriented
cycles, which can be viewed as one-point extensions. It was proved
that in this case each Lie derivation is of the standard form.
Moreover, the standard decomposition is unique. On the other hand,
we remark that its dual extension algebra of arbitrary finite
dimensional algebra inherit many wonderful properties from the given
algebra. Then for the path algebra of a finite quiver without
oriented cycles, it is natural to ask whether all Lie derivations on
the dual extension algebra are of the standard form. We will give a
positive answer in this paper. More
precisely, the main result of this paper is\\

\noindent{\bf Theorem.} \,\,{\em Let $K$ be a field of
characteristic not $2$. Let $(\Gamma, \rho)$ be a finite quiver
without oriented cycles. Then each Lie derivation on the dual
extension of the path algebra $K(\Gamma,\rho)$ is of the standard
form $(\spadesuit)$.
Moreover, the standard decomposition is unique.}\\

Jordan derivations, another important class of linear mappings on
dual extension algebras, has been characterized in \cite{LiWei3},
where we show that every Jordan derivation on dual extension
algebras is also of the standard form.

Note that each associative algebra with non trivial idempotents is
isomorphic a generalized matrix algebra. Recently, Du and Wang
\cite{DuWang} studied Lie derivations of generalized matrix algebras
with bi-modules $M$ being faithful. Although the methods of matrix
algebras is also employed in our current work, we prefer to the
assumptions without faithful conditions. When Cheung \cite{Cheung}
investigated Lie derivations of triangular algebras, the faithful
assumption is not needed. In this sense, Section $3$ of this paper
is a natural generalization of Cheung's work. Simultaneously, our
work is an attempt to deal with path algebras of quivers with
oriented cycles. So this article is also a continuation and
development of \cite{LiWei2}.

The paper is organized as follows. After a rapid review of some
needed preliminaries in Section $2$, we characterize Lie derivations
of generalized matrix algebras in Section $3$. We study Lie
derivations of dual extensions in Section $4$, where the main result
of this paper will be eventually obtained.

\section{Dual extension}\label{xxsec2}

Let us first recall the definition of dual extensions of path
algebras which were introduced by Xi \cite{Xi1}. Moreover, in order
to use the methods of matrix algebras, we will also give some
descriptions of dual extensions from the point of view of
generalized matrix algebras. This kind of algebra was introduced by
Morita in \cite{Morita}, where the author studied Morita duality
theory of modules and its applications to Artinian algebras. Let us
begin with the definition of generalized matrix algebras.

\subsection{Generalized matrix algebras}\label{xxsec2.1}

The definition of generalized matrix algebras is given by a Morita
context. Let $\mathcal{R}$ be a commutative ring with identity. A
Morita context consists of two $\mathcal{R}$-algebras $A$ and $B$,
two bimodules $_AM_B$ and $_BN_A$, and two bimodule homomorphisms
called the pairings $\Phi_{MN}: M\underset {B}{\otimes}
N\longrightarrow A$ and $\Psi_{NM}: N\underset {A}{\otimes}
M\longrightarrow B$ satisfying the following commutative diagrams:
$$
\xymatrix{ M \underset {B}{\otimes} N \underset{A}{\otimes} M
\ar[rr]^{\hspace{8pt}\Phi_{MN} \otimes I_M} \ar[dd]^{I_M \otimes
\Psi_{NM}} && A
\underset{A}{\otimes} M \ar[dd]^{\cong} \\  &&\\
M \underset{B}{\otimes} B \ar[rr]^{\hspace{10pt}\cong} && M }
\hspace{4pt}$$ and
$$\xymatrix{ N \underset
{A}{\otimes} M \underset{B}{\otimes} N
\ar[rr]^{\hspace{8pt}\Psi_{NM}\otimes I_N} \ar[dd]^{I_N\otimes
\Phi_{MN}} && B
\underset{B}{\otimes} N \ar[dd]^{\cong}\\  &&\\
N \underset{A}{\otimes} A \ar[rr]^{\hspace{10pt}\cong} && N
\hspace{2pt}.}
$$
Let us write this Morita context as $(A, B, _AM_B, _BN_A, \Phi_{MN},
\Psi_{NM})$. If $(A, B, _AM_B,$ $ _BN_A,$ $ \Phi_{MN}, \Psi_{NM})$
is a Morita context, then the set
$$
\left[
\begin{array}
[c]{cc}%
A & M\\
N & B\\
\end{array}
\right]=\left\{ \left[
\begin{array}
[c]{cc}%
a& m\\
n & b\\
\end{array}
\right] \vline a\in A, m\in M, n\in N, b\in B \right\}
$$
form an $\mathcal{R}$-algebra under matrix-like addition and
matrix-like multiplication. There is no constraint condition
concerning the bimodules $M$ and $N$. Of course, they probably equal
to zeros. Such an $\mathcal{R}$-algebra is called a
\textit{generalized matrix algebra} of order 2 and is usually
denoted by $\mathcal{G}=\left[
\begin{array}
[c]{cc}%
A & M\\
N & B
\end{array}
\right]$. Its center $\mathcal {Z}(\mathcal{G})$ is
$$
\mathcal {Z}(\mathcal{G})=\left\{ \left[
\begin{array}
[c]{cc}%
a & 0\\
0 & b
\end{array}
\right] \vline \hspace{2pt}a\in \mathcal {Z}(A), b\in \mathcal
{Z}(B), am=mb,\,\,na=bn, \ \forall\ m\in M,\,\, n\in N \right\}.
$$
Thus we have two natural $\mathcal{R}$-linear projections
$\pi_A:\mathcal{G}\rightarrow A$ and $\pi_B:\mathcal{G}\rightarrow
B$ by
$$
\pi_A: \left[
\begin{array}
[c]{cc}%
a & m\\
n & b\\
\end{array}
\right] \longmapsto a \quad \text{and} \quad \pi_B: \left[
\begin{array}
[c]{cc}%
a & m\\
n & b\\
\end{array}
\right] \longmapsto b.
$$
Then $\pi_A \left(\mathcal{Z(G)}\right)$ is a subalgebra of
$\mathcal{Z}(A)$ and that $\pi_B\left(\mathcal{Z(G)}\right)$ is a
subalgebra of $\mathcal{Z}(B)$. If $M$ is faithful as a right
$B$-module and as left $A$-module, then for every element
$a\in\pi_A(\mathcal{Z(G)})$, there exists a unique
$b\in\pi_B(\mathcal{Z(G)})$, which is denoted by $\varphi(a)$,
such that $
\left[\smallmatrix a & 0\\
0 & b
\endsmallmatrix \right] \in \mathcal{Z(G)}$. It is easy to
verify that the mapping
$\varphi:\pi_A(\mathcal{Z(G)})\longrightarrow \pi_B(\mathcal{Z(G)})$
is an algebraic isomorphism such that $am=m\varphi(a)$ and
$na=\varphi(a)n$ for all $a\in \pi_A(\mathcal{Z(G)}), m\in M, n\in
N$.

\begin{remark}\label{2.1}
Any unital $\mathcal{R}$-algebra $\mathcal {A}$ with nontrivial
idempotents is isomorphic to a generalized matrix algebra. In fact,
suppose that there exists a nontrivial idempotent $e\in \mathcal
{A}$. We construct the following \textit{natural generalized matrix
algebra}:
\begin{align*}
\mathcal{G} & = \left[
\begin{array}
[c]{cc}%
e\mathcal {A}e & e\mathcal {A}(1-e)\\
(1-e)\mathcal {A}e & (1-e)\mathcal {A}(1-e)\\
\end{array}
\right]\\ & =\left\{ \hspace{2pt} \left[
\begin{array}
[c]{cc}%
eae & ec(1-e)\\
(1-e)de & (1-e)b(1-e)\\
\end{array}
\right] \hspace{3pt} \vline \hspace{3pt} a, b, c, d\in \mathcal
{A}\hspace{3pt} \right\} .
\end{align*}
It is easy to check that the $\mathcal{R}$-linear mapping
\begin{align*}
\xi: \mathcal {A} & \longrightarrow \mathcal{G} \\
a & \longmapsto \left[
\begin{array}
[c]{cc}%
eae & ea(1-e)\\
(1-e)ae & (1-e)a(1-e)\\
\end{array}
\right]
\end{align*}
is an isomorphism from $\mathcal {A}$ to $\mathcal{G}$.
\end{remark}

\subsection{Dual extension of path algebras}
Recall that a \textit{finite quiver} $\Gamma$ is an oriented graph
with the set of vertices $\Gamma_0$ and the set of arrows between
vertices $\Gamma_1$ being both finite. For an arrow $\alpha$, we
write $s(\alpha)=i$ and $e(\alpha)=j$ if it is from the vertex $i$
to the vertex $j$. A \textit{sink} is a vertex without arrows
beginning at it and a \textit{source} is a vertex without arrows
ending at it. A \textit{nontrivial path} in $\Gamma$ is an ordered
sequence of arrows $p=\alpha_n\cdots\alpha_1$ such that
$e(\alpha_m)=s(\alpha_{m+1})$ for each $1\leq m<n$. Define
$s(p)=s(\alpha_1)$ and $e(p)=e(\alpha_n)$. A \textit{trivial path}
is the symbol $e_i$ for each $i\in \Gamma_0$. In this case, we set
$s(e_i)=e(e_i)=i$. A nontrivial path $p$ is called an
\textit{oriented cycle} if $s(p)=e(p)$. Denote the set of all paths
by $\mathscr{P}$.

Let $K$ be a field and $\Gamma$ be a quiver. Then the path algebra
$K\Gamma$ is the $K$-algebra generated by the paths in $\Gamma$
and the product of two paths $x=\alpha_n\cdots\alpha_1$ and
$y=\beta_t\cdots\beta_1$ is defined by
$$
xy=\left\{
\begin{array}{ll}
\alpha_n\cdots\alpha_1\beta_t\cdots\beta_1, & \mbox{$e(y)=s(x)$};\\
0, & \mbox{otherwise}.
\end{array}
\right.
$$
Clearly, $K\Gamma$ is an associative algebra with the identity
$1=\sum_{i\in \Gamma_0}e_i$, where $e_i(i\in \Gamma_0)$ are
pairwise orthogonal primitive idempotents of $K\Gamma$.

A {\em relation} $\sigma$ on a quiver $\Gamma$ over a field $K$ is
a $K$-linear combination of paths $\sigma=\sum_{i=1}^nk_ip_i,$
where $k_i\in K$ and
$$e(p_1)=\cdots=e(p_n),\,\,\, s(p_1)=\cdots=s(p_n).$$ Moreover, the
number of arrows in each path is assumed to be at least 2. Let
$\rho$ be a set of relations on $\Gamma$ over $K$. The pair
$(\Gamma, \rho)$ is called a \textit{quiver} with relations over
$K$. Denote by $<\rho>$ the ideal of $K\Gamma$ generated by the set
of relations $\rho$. The $K$-algebra $K(\Gamma,
\rho)=K\Gamma/<\rho>$ is always associated with $(\Gamma, \rho)$.
For arbitrary element $x\in K\Gamma$, write by $\overline x$ the
corresponding element in $K(\Gamma, \rho)$. We often write
$\overline x$ as $x$ if this is not misled or confused. We refer the
reader to \cite{AuslanderReitenSmalso} for the basic facts of path
algebras.

Let $\Lambda=K(\Gamma, \rho)$, where $\Gamma$ is a finite quiver.
Let $\Gamma^{\ast}$ to be a quiver whose vertex set is $\Gamma_0$
and
$$
\Gamma_1^{\ast}=\{\alpha^{\ast}: i\rightarrow j\mid \alpha: j\rightarrow i
\,\,\,\text{is an arrow in} \,\,\,\Gamma_1\}.
$$ Let
$p=\alpha_n\cdots\alpha_1$ is a path in $\Gamma$. Write the path
$\alpha_1^{\ast}\cdots\alpha_n^{\ast}$ in $\Gamma^{\ast}$ by
$p^{\ast}$. Define $\mathscr{D}(\Lambda)$ to be the path algebra of
the quiver $(\Gamma_0, \Gamma_1\cup\Gamma_1^{\ast})$ with relations
$$
\rho\,\,\cup\,\,\rho^{\ast}\,\,\cup\,\,\{\alpha\beta^{\ast}\mid\alpha,\beta\in\Gamma_1\}.
$$
If $\Gamma$ has no oriented cycles, then $\mathscr{D}(\Lambda)$ is
called the {\em dual extension} of $\Lambda$. It is a BGG-algebra in
the sense of \cite{Irving}. Clearly, if $|\Gamma_0|=1$, then the
algebra is trivial. Let us assume that $|\Gamma_0|\geq 2$ from now
on. It is helpful to point out that in this case,
$\mathscr{D}(\Lambda)$ has non-trivial idempotents. In view of
Remark \ref{2.1}, $\mathscr{D}(\Lambda)$ is isomorphic to a
generalized matrix algebra
$\mathcal{G}=\left[\smallmatrix A & M\\
N & B \endsmallmatrix \right]$.

Let us take the nontrivial idempotent to be $e_i$, where $i$ is a
source of $\Gamma$. According to the definition of dual extension,
it is easy to verify that the pairings $\Phi_{MN}=0$ and
$\Psi_{NM}\neq 0$. If $M\neq 0$, then $N\neq 0$. Moreover, it is
helpful to point out that $M$ need not to be faithful as left
$A$-module or as right $B$-module. Let us illustrate two examples in
below.

\begin{example}
Let $\Gamma$ be a quiver as follows
$$
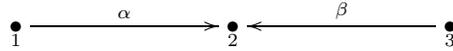
\xymatrix@C=25mm{
  \bullet
  \ar@<0pt>[r]_(0){1}^{\alpha}  &
  \bullet &
  \bullet
  \ar@<0pt>@[r][l]_(0.5){\beta}^(0){3}^(1){2} }$$
and let $\Lambda=K\Gamma$. The dual extension
$\mathscr{D}(\Lambda)$ has a basis $$\{e_1, e_2, e_3, \alpha,
\beta, \alpha^{\ast}, \beta^{\ast}, \alpha^{\ast}\alpha,
\beta^{\ast}\alpha, \beta^{\ast}\beta, \alpha^{\ast}\beta\}.$$
Taking the nontrivial idempotent to be $e_1$, then
$\mathscr{D}(\Lambda)$ is isomorphic to the generalized matrix
algebra $\mathcal {G}=\left[
\begin{array}
[c]{cc}%
A & M\\
N & B\\
\end{array}
\right]$, where $A$ has a basis $\{e_2, e_3, \beta, \beta^{\ast},
\beta^{\ast}\beta\}$, $B$ has a basis $\{e_1, \alpha^{\ast}\alpha
\}$, $M$ has a basis $\{\alpha, \beta^{\ast}\alpha\}$ and $N$ has a
basis $\{\alpha^{\ast}, \alpha^{\ast}\beta\}$. It follows from
$\beta\alpha=0$ and $\beta\beta^{\ast}\alpha=0$ that $\beta\in {\rm
anni}(_AM)$. That is, $M$ is not faithful as left $A$-module. It is
easy to check that $\alpha^{\ast}\alpha\in {\rm anni}(M_B)$. This
implies that $M$ is not faithful as right $B$-module. Similarly, we
obtain $\alpha^{\ast}\alpha\in {\rm anni}(_BN)$ and
$\beta^{\ast}\beta\in {\rm anni}(N_A)$. That is, $N$ is neither a
faithful left $B$-module nor a faithful right $A$-module.
\end{example}

\begin{example}
Let $\Gamma$ be a quiver as follows

$$
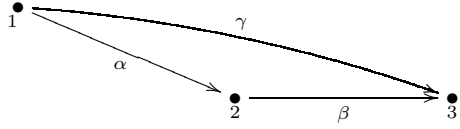
\xymatrix@C=25mm{
  \bullet \ar@/^/[drr]^{\gamma}
  \ar@{>}[dr]_(0){1}_{\alpha}\\ &
  \bullet\ar[r]_(0){2}_(1){3}_{\beta} & \bullet
  }$$
and let $\Lambda=K\Gamma$. Taking the nontrivial idempotent to be
$e_1$, then the dual extension $\mathscr{D}(\Lambda)\simeq \mathcal
{G}(A, M, N, B)$, where $A$ has a basis $\{e_2, e_3, \beta,
\beta^{\ast},\beta^{\ast}\beta\}$, $M$ has a basis $\{\alpha,
\beta^{\ast}\beta\alpha, \beta^{\ast}\gamma, \beta\alpha, \gamma\}$,
$N$ has a basis $\{\alpha^{\ast}, \gamma^{\ast}\beta,
\alpha^{\ast}\beta^{\ast}\beta, \alpha^{\ast}\beta^{\ast},
\gamma^{\ast}\}$, $B$ has a basis $\{e_1, \alpha^{\ast}\alpha,
\gamma^{\ast}\gamma, \alpha^{\ast}\beta^{\ast}\gamma,
\gamma^{\ast}\beta\alpha, \alpha^{\ast}\beta^{\ast}\beta\alpha\}$.
Clearly, $e_2\alpha\neq 0$, $e_3\beta\alpha\neq 0$. Then $e_2,
e_3\notin {\rm anni}(M)$. Similarly, $\beta\alpha\neq 0$ implies
that $\beta\notin {\rm anni}(M)$; $\beta^{\ast}\beta\alpha\neq 0$
implies that $\beta^{\ast}\notin {\rm anni}(M)$ and
$\beta^{\ast}\beta\notin {\rm anni}(M)$. Hence $M$ is faithful as a
left $A$-module. On the other hand, it is easy to check that
$\alpha^{\ast}\beta^{\ast}\beta\alpha\in {\rm anni}(M_B)$. Thus $M$
is not faithful as a right $B$-module. Similarly, we know that $N$
is faithful as a right $A$-module, while it is not faithful as a
left $B$-module.
\end{example}

\begin{remark}
In order to study global dimension of dual extensions, one more
general definition of dual extension algebras was posed by Xi
\cite{Xi2}. We omit the details here because it will not be used in
our current work.
\end{remark}

\section{Lie derivations of generalized matrix algebras}\label{xxsec3}

In Section $2$ we have pointed out that the dual extension of a path
algebra can be viewed as a generalized matrix algebra. In order to
study Lie derivations of dual extension algebras, it is necessary to
provide some basic facts concerning Lie derivations of generalized
matrix algebras. In this section, we will give a sufficient and
necessary condition which enable every Lie derivation to be standard
$(\spadesuit)$.

From now on, we always assume that all algebras and bimodules are
$2$-torsion free. Note that the forms of derivations and Lie
derivations of a generalized matrix algebra have already been
described in \cite{LiWei1}.

\begin{lemma}\cite[Proposition 4.1]{LiWei1}\label{3.1}
Let $\Theta_{\rm Lied}$ be a Lie derivation of a generalized matrix
algebra $\mathcal{G}=\left[
\begin{array}
[c]{cc}%
A & M\\
N & B\\
\end{array}
\right]$. Then $\Theta_{\rm Lied}$ is of the form
$$
\begin{aligned}
& \Theta_{\rm Lied}\left(\left[
\begin{array}
[c]{cc}%
a & m\\
n & b\\
\end{array}
\right]\right) \\
=& \left[
\begin{array}
[c]{cc}%
\delta_1(a)-mn_0-m_0n+\delta_4(b) & am_0-m_0b+\tau_2(m)\\
n_0a-bn_0+\nu_3(n) & \mu_1(a)+n_0m+nm_0+\mu_4(b)\\
\end{array}
\right] ,\\
& \forall \left[
\begin{array}
[c]{cc}%
a & m\\
n & b\\
\end{array}
\right]\in \mathcal{G},
\end{aligned}
$$
where $m_0\in M, n_0\in N$ and
$$
\begin{aligned} \delta_1:& A \longrightarrow A, &
\delta_4: & B\longrightarrow \mathcal{Z}(A) &  \tau_2:
& M\longrightarrow M, \\
\nu_3: & N\longrightarrow N & \mu_1: & A\longrightarrow
\mathcal{Z}(B), & \mu_4: & B\longrightarrow B
\end{aligned}
$$
are all $\mathcal{R}$-linear mappings satisfying the following
conditions:
\begin{enumerate}
\item[{\rm(1)}] $\delta_1$ is a Lie derivation of $A$ and
$\delta_1(mn)=\delta_4(nm)+\tau_2(m)n+m\nu_3(n);$

\item[{\rm(2)}] $\mu_4$ is a Lie derivation of $B$ and
$\mu_4(nm)=\mu_1(mn)+n\tau_2(m)+\nu_3(n)m;$

\item[{\rm(3)}] $\delta_4([b,b'])=0$ for all $b, b^\prime\in B$ and
$\mu_1([a,a'])=0$ for all $a, a^\prime\in A;$

\item[{\rm(4)}] $\tau_2(am)=a\tau_{2}(m)+\delta_1(a)m-m\mu_1(a)$ and
$\tau_2(mb)=\tau_2(m)b+m\mu_4(b)-\delta_4(b)m;$

\item[{\rm(5)}] $\nu_3(na)=\nu_3(n)a+n\delta_1(a)-\mu_1(a)n$ and
$\nu_3(bn)=b\nu_3(n)+\mu_4(b)n-n\delta_4(b).$
\end{enumerate}
\end{lemma}

\begin{lemma}\cite[Proposition 4.2]{LiWei1}\label{3.2}
An additive mapping $\Theta_{\rm d}$ is a derivation of
$\mathcal{G}$ if and only if $\Theta_d$ has the form
$$
\begin{aligned}
& \Theta_{\rm d}\left(\left[
\begin{array}
[c]{cc}%
a & m\\
n & b\\
\end{array}
\right]\right) \\
=& \left[
\begin{array}
[c]{cc}%
\delta_1(a)-mn_0-m_0n & am_0-m_0b+\tau_2(m)\\
n_0a-bn_0+\nu_3(n) & n_0m+nm_0+\mu_4(b)\\
\end{array}
\right] ,\\
& \forall \left[
\begin{array}
[c]{cc}%
a & m\\
n & b\\
\end{array}
\right]\in \mathcal{G},
\end{aligned}
$$
where $m_0\in M, n_0\in N$ and
$$
\begin{aligned} \delta_1:& A \longrightarrow A, &
 \tau_2: & M\longrightarrow M, & \tau_3: & N\longrightarrow M,\\
\nu_2: & M\longrightarrow N, & \nu_3: & N\longrightarrow N , &
\mu_4: & B\longrightarrow B
\end{aligned}
$$
are all $\mathcal{R}$-linear mappings satisfying the following
conditions:
\begin{enumerate}
\item[{\rm(1)}] $\delta_1$ is a derivation of $A$ with
$\delta_1(mn)=\tau_2(m)n+m\nu_3(n);$

\item[{\rm(2)}] $\mu_4$ is a derivation of $B$ with
$\mu_4(nm)=n\tau_2(m)+\nu_3(n)m;$

\item[{\rm(3)}] $\tau_2(am)=a\tau_{2}(m)+\delta_1(a)m$ and
$\tau_2(mb)=\tau_2(m)b+m\mu_4(b);$

\item[{\rm(4)}] $\nu_3(na)=\nu_3(n)a+n\delta_1(a)$ and
$\nu_3(bn)=b\nu_3(n)+\mu_4(b)n;$
\end{enumerate}
\end{lemma}

In \cite{Cheung}, Cheung gave a necessary and sufficient condition
such that each Lie derivation on a triangular algebra has the
standard form $(\spadesuit)$. We next extend it to the generalized
matrix algebras context.

\begin{theorem}\label{3.3}
Let $\Theta_{\rm Lied}$ be a Lie derivation of a generalized matrix
algebra $\mathcal{G}$. Then $\Theta_{\rm Lied}$ is of the standard
form $(\spadesuit)$ if and only if there exist linear mappings
$l_A:A\rightarrow {\mathcal Z}(A)$ and $l_B:B\rightarrow {\mathcal
Z}(B)$ satisfying
\begin{enumerate}
\item[{\rm(1)}] $p_A=\delta_1-l_A$ is a derivation on $A$,
$l_A([a,a'])=0$, $l_A(mn)=\delta_4(nm)$, $l_A(a)m=m\mu_1(a)$ and
$nl_A(a)=\mu_1(a)n$.

\item[{\rm(2)}] $p_B=\mu_4-l_B$ is a derivation on $B$,
$l_B([b,b'])=0$, $l_B(nm)=\mu_1(mn)$, $l_B(b)n=n\delta_4(b)$ and
$ml_B(b)=\delta_4(b)m$.
\end{enumerate}
\end{theorem}

\begin{proof}
Let us first prove the necessity. Suppose that $\Theta_{\rm
Lied}=\delta+h$, where $\delta$ is a derivation and $h$ maps
$\mathcal{G}$ into ${\mathcal Z}(\mathcal{G})$. Then by Lemma
\ref{3.2}, there exist linear mappings $l_A:A\rightarrow A$ and
$l_B:B\rightarrow B$ such that $p_A=\delta_1-l_A$ is a derivation of
$A$ and $p_B=\mu_4-l_B$ is a derivation of $B$. This gives
$$
h\left(\left[
\begin{array}
[c]{cc}%
a & m\\
n & b\\
\end{array}
\right]\right)=\left[
\begin{array}
[c]{cc}%
l_A(a)+\delta_4(b) & 0\\
0 & \mu_1(a)+l_B(b)\\
\end{array}
\right]\in Z(\mathcal {G}), \forall \left[
\begin{array}
[c]{cc}%
a & m\\
n & b\\
\end{array}
\right]\in \mathcal{G}.
$$
By Lemma \ref{3.1} we know that $\delta_4(b)\in {\mathcal Z}(A)$,
$\mu_1(a)\in {\mathcal Z}(B)$ for all $b\in B$ and $a\in A$. Then
the structure of ${\mathcal Z}(\mathcal{G})$ implies that $l_A$ maps
into ${\mathcal Z}(A)$ and $l_B$ maps into ${\mathcal Z}(B)$.
Furthermore, $l_A(a)m=m\mu_1(a)$, $nl_A(a)=\mu_1(a)n$,
$l_B(b)n=m\delta_4(b)$ and $ml_B(b)=\delta_4(b)m$ are also follow.

Note that $h$ is also a Lie derivation of $\mathcal{G}$. In view of
Lemma \ref{3.1} we have that $l_A(mn)=\delta_4(nm)$
and $l_B(nm)=\mu_1(mn)$ for all $m\in M, n\in N$.
Let us choose $G_1=\left[\smallmatrix a & 0\\
0 & 0 \endsmallmatrix \right]$ and $G_2=\left[\smallmatrix a' & 0\\
0 & 0 \endsmallmatrix \right]$ and take them into
$$
h([G_1, G_2])=[h(G_1), G_2]+[G_1, h(G_2)]. \eqno(3.1)
$$
It is not difficult to calculate that
$$
\begin{aligned}
h([G_1, G_2])=& \left[
\begin{array}
[c]{cc}%
l_A([a, a']) & 0\\
0 & \mu_1([a, a'])\\
\end{array}
\right]
\end{aligned}
\eqno(3.2)
$$
and
$$
\begin{aligned}
&[h(G_1), G_2]+[G_1, h(G_2)]= \left[
\begin{array}
[c]{cc}%
[l_A(a), a']+[a, l_A(a')] & 0\\
0 & 0\\
\end{array}
\right].
\end{aligned}
\eqno(3.3)
$$
It follows from $l_A(a), l_A(a')\in {\mathcal Z}(A)$ that $[l_A(a),
a']+[a, l_A(a')]=0$. Combining $(3.2)$ with $(3.3)$ yields that
$l_A([a, a'])=0$. Similarly, if we take $G_1=\left[\smallmatrix 0 & 0\\
0 & b \endsmallmatrix \right]$ and $G_2=\left[\smallmatrix 0 & 0\\
0 & b' \endsmallmatrix \right]$ in $(3.1)$, then $l_B([b,b'])=0$
will be obtained.

\smallskip\smallskip

Let us see the sufficiency. Let us take
$$
\delta\left(\left[
\begin{array}
[c]{cc}%
a & m\\
n & b\\
\end{array}
\right]\right)=\left[
\begin{array}
[c]{cc}%
p_A(a)-mn_0-m_0n & am_0-m_0b+\tau_2(m)\\
n_0a-bn_0+\nu_3(n) & n_0m+nm_0+p_B(b)\\
\end{array}
\right], \forall \left[
\begin{array}
[c]{cc}%
a & m\\
n & b\\
\end{array}
\right]\in \mathcal{G}
$$
and
$$
h\left(\left[
\begin{array}
[c]{cc}%
a & m\\
n & b\\
\end{array}
\right]\right)=\left[
\begin{array}
[c]{cc}%
l_A(a)+\delta_4(b) & 0\\
0 & \mu_1(a)+l_B(b)\\
\end{array}
\right], \forall \left[
\begin{array}
[c]{cc}%
a & m\\
n & b\\
\end{array}
\right]\in \mathcal{G}.
$$
It is easy to verify that $\delta$ is a derivation of $\mathcal{G}$
and $h$ maps into ${\mathcal Z}(\mathcal{G})$. For arbitrary $G=\left[\smallmatrix a & m\\
n & b \endsmallmatrix \right]$ and $G'=\left[\smallmatrix a' & m'\\
n' & b' \endsmallmatrix \right]$, we have
$$h([G, G'])=\left[
\begin{array}
[c]{cc}%
l_A(x+u)+\delta_4(v+y) & 0\\
0 & \mu_1(x+u)+l_B(v+y)\\
\end{array}
\right],$$ where $x=[a, a']$, $y=[b, b']$, $u=mn'-m'n$ and
$v=nm'-n'm$. Note that condition (1) implies that
$l_A(x+u)+\delta_4(v+y)=0$ and condition (2) implies that
$\mu_1(x+u)+l_B(v+y)=0$. Therefore $h([G, G'])=0$.
\end{proof}

The following corollary provides us a sufficient condition which
enable each Lie derivation of $\mathcal {G}(A, M, N, B)$ to be
standard, where $M$ is faithful as left $A$-module and as right
$B$-module.

\begin{corollary}\label{3.4}
Let $\mathcal {G}=\left[
\begin{array}
[c]{cc}%
A & M\\
N & B\\
\end{array}
\right]$ be a generalized matrix algebra. Suppose that $M$ is
faithful as a left $A$-module and is also faithful as a right
$B$-module. If ${\mathcal Z}(A)=\pi_A({\mathcal Z}(\mathcal {G}))$
and ${\mathcal Z}(B)=\pi_B({\mathcal Z}(\mathcal {G}))$, then every
Lie derivation of $\mathcal {G}$ has the standard form
$(\spadesuit)$.
\end{corollary}

\begin{proof}
Let $\Theta$ be a Lie derivation of $\mathcal {G}$ with the form
described in Lemma \ref{3.1}. Then it follows from ${\mathcal
Z}(A)=\pi_A({\mathcal Z}(\mathcal {G}))$ and ${\mathcal
Z}(B)=\pi_B({\mathcal Z}(\mathcal {G}))$ and Lemma \ref{3.1} that
$$
\begin{aligned}
& h\left(\left[
\begin{array}
[c]{cc}%
a & m\\
n & b\\
\end{array}
\right]\right)  =& \left[
\begin{array}
[c]{cc}%
\varphi^{-1}(\mu_1(a))+\delta_4(b) & 0\\
0 & \mu_1(a)+\varphi(\delta_4(b))\\
\end{array}
\right]\in Z(\mathcal {G}).
\end{aligned}
$$
On the other hand, it is a direct computation that $\Theta-h$ is a
derivation of $\mathcal {G}$. This completes the proof.
\end{proof}

\begin{remark}\label{3.5}
Du and Wang obtained one much more general version of Corollary
\ref{3.4} in \cite{DuWang}. We omit the details here.
\end{remark}

\begin{corollary}\label{3.6}
Let $\mathcal{U}=\left[
\begin{array}
[c]{cc}%
A & M\\
O & B\\
\end{array}
\right]$ be a triangular algebra. Suppose that $M$ is faithful as a
left $A$-module and is also faithful as a right $B$-module. If
${\mathcal Z}(A)=\pi_A({\mathcal Z}(\mathcal {U}))$ and ${\mathcal
Z}(B)=\pi_B({\mathcal Z}(\mathcal {U}))$, then every Lie derivation
of $\mathcal {U}$ is of the standard form $(\spadesuit)$.
\end{corollary}

\smallskip\smallskip

Let us continue to develop Theorem $11$ of \cite{Cheung} to the case
of generalized matrix algebras. As in \cite{Cheung}, we need to give
two preliminary lemmas.

\begin{lemma}\label{3.7}
Let $\delta_1$ be a Lie derivation of $A$ and let $\mu_1:
A\rightarrow \mathcal {Z}(B)$ and $\tau_2: M\rightarrow M$, $\nu_3:
N\rightarrow N$ be linear mappings satisfying
$$
\tau_2(am)=a\tau_2(m)+\delta_1(a)m-m\mu_1(a)
$$ and
$$
\nu_3(na)=\nu_3(n)a+n\delta_1(a)-\mu_1(a)n.
$$ Define $G: A\times
A\rightarrow A$ by
$$
G(x,
y)=\delta_1(xy)-x\delta_1(y)-\delta_1(x)y.
$$
Then the following statements hold:
\begin{enumerate}
\item[(1)] $G(x, y)=G(y, x)$;
\item[(2)] $G(x, y)m=m\mu_1(xy)-xm\mu_1(y)-ym\mu_1(x)$; \\$nG(x,
y)=\mu_1(xy)n-\mu_1(y)nx-\mu_1(x)ny$.
\item[(3)] Let $f(t)=\sum_{j=0}^k r_jt^j\in K[t]$ and $x\in A$. Then
there exists $a_x\in A$ such that
$$a_xm=m\mu_1(f(x))-f'(x)m\mu_1(x)$$ for all $m\in M$ and
$$na_x=\mu_1(f(x))n-\mu_1(x)nf'(x)$$  for all $n\in N$, where
$f'(t)=\sum_{j=1}^kjr_jt^{j-1}$.
\end{enumerate}

Moreover, if $\delta_1$ is of the standard form, that is,
$\delta_1=p_A\,+\,l_A$, where $p_A$ is a derivation of $A$ and $l_A$
maps into ${\mathcal Z}(A)$, then $a_x=l_A(f(x))-f'(x)l(x))$.
\end{lemma}

\begin{proof}
The relation (1) and the first one of (2) can be obtained via (i)
and (ii) of Lemma 9 in \cite{Cheung}. It suffices to prove the
second equality of (2). In fact, we have from the relation
$\nu_3(na)=\nu_3(n)a+n\delta_1(a)-\mu_1(a)n$ that
$$\nu_3(n(xy))=\nu_3(n)xy+n\delta_1(xy)-\mu_1(xy)n$$ and
$$
\begin{aligned}
\nu_3((nx)y)&=\nu_3(nx)y+nx\delta_1(y)-\mu_1(y)nx\\
&=\nu_3(n)xy+n\delta_1(x)y-\mu_1(x)ny
\end{aligned}
$$
Then comparing the above two equalities gives the required result.

In order to prove (3), it is enough to consider $f(t)=t^k$, where
$k=0, 1, 2, \cdots$. For $k=0$, we can take $a_x=\delta_1(1)$, which
is due to the conditions (4) and (5) of Lemma \ref{3.1}. For $k>0$,
let us take
$$
a_x=\sum_{j=1}^{k-1}x^{k-1-j}G(x^j, x).
$$
Then the first equality
of (2) implies that
$$
a_xm=m\mu_1(f(x))-f'(x)m\mu_1(x).
$$
The second equality of (2) implies that $$
na_x=\mu_1(f(x))n-\mu_1(x)nf'(x).
$$
\end{proof}

\begin{lemma}\label{3.8}
Assume that $\delta_1=p_A+l_A$, where $p_A$ is a derivation and
$l_A(a)\in \mathcal {Z}(A)$, $l_A([a, a'])=0$ for all $a, a'\in
A$. Let $$V_A=\{a\in A\mid l_A(a)m=m\mu_1(a),
\,\,nl_A(a)=\mu_1(a)n\,\,\,\, \forall \,\,\,m\in M,\,\,
\forall\,\,\,n\in N\}.$$ Then $V_A$ is a subalgebra of $A$
satisfying the following conditions.
\begin{enumerate}
\item[(1)] $[x, y]\in V_A$ for all $x, y\in A$.
\item[(2)] Let $f(t)\in K[t]$ and $x\in A$. If $f'(x)=0$, then $f(x)\in
V_A$.
\item[(3)] $V_A$ contains all the idempotents of $A$.
\end{enumerate}
\end{lemma}

\begin{proof}
For arbitrary $x, y\in V_A$, we have from (2) of Lemma \ref{3.7}
that
$$
\begin{aligned}
m\mu_1(xy)&=G(x, y)m+xm\mu_1(y)+ym\mu_1(x)\\
&=\delta_1(xy)m-x\delta_1(y)m-\delta_1(x)ym+xl_A(y)m+yl_A(x)m\\
&=p_A(xy)m+l_A(xy)m-xp_A(y)m\\
&\hspace{10pt}-xl_A(y)m-p_A(x)ym-l_A(x)ym+xl_A(y)m+yl_A(x)m
\end{aligned}
$$
Note that $p_A$ is a derivation and $l_A(x)\in \mathcal {Z}(A)$.
Then we obtain $$m\mu_1(xy)=l_A(xy)m. \eqno(3.4)$$ On the other
hand,
$$
\begin{aligned}
\mu_1(xy)n&=nG(x, y)+\mu_1(y)nx+\mu_1(x)ny\\
&=n\delta_1(xy)-nx\delta_1(y)-n\delta_1(x)y+nl_A(y)x+nl_A(x)y\\
&=np_A(xy)+nl_A(xy)-nxp_A(y)\\
&\hspace{10pt}-nxl_A(y)-np_A(x)y-nl_A(x)y+nl_A(y)x+nl_A(x)y\\
&=nl_A(xy),
\end{aligned}
$$
That is,
$$
\mu_1(xy)n=nl_A(xy). \eqno(3.5)
$$
Combining $(3.4)$ with $(3.5)$ yields that $V_A$ is a subalgebra of
$A$.

We now prove the other three conditions. Clearly, (1) follows from
that $\mu_1$ annihilates all commutators.

In order to prove (2), let us take $x\in A$ with $f'(x)=0$ and
$f(t)\in \mathcal {R} [t]$. In view of (3) of Lemma \ref{3.7}, we
know that there exists $a_x\in A$ such that $a_xm=m\mu_1(f(x))$ and
$na_x=\mu_1(f(x))n$. Since $\delta_1$ is of the standard form,
$a_xm=l_A(f(x))m$ and $na_x=nl_A(f(x))$ by Lemma \ref{3.7}.
Therefore $l_A(f(x))m=m\mu_1(f(x))$ and $nl_A(f(x))=\mu_1(f(x))n$.
That is, $f(x)\in V_A$.

The proof of (3) is the same as that of (4) of Lemma 10 in
\cite{Cheung}. For the convenience of readers, we copy it here. In
fact, for arbitrary idempotent $e\in A$, let us put
$f(t)=3t^2-2t^3$. Then clearly, $f'(e)=0$. Hence $e=f(e)\in V_A$.
\end{proof}

As in \cite{Cheung}, we define $W(X)$ to be the smallest subalgebra
of an algebra $X$ satisfying conditions (1)-(3) of Lemma \ref{3.8}.
Then the following result is a direct consequence of Theorem
\ref{3.3}.

\begin{corollary}\label{3.9}
Let $\mathcal{G}=\left[
\begin{array}
[c]{cc}%
A & M\\
N & B\\
\end{array}
\right]$ be a generalized matrix algebra with zero bilinear
pairings. If
\begin{enumerate}
\item[(1)] $W(A)=A$ and every Lie derivation of $A$ is of the
standard form $(\spadesuit)$;
\item[(2)] $W(B)=B$ and every Lie derivation of $B$
is of the standard form $(\spadesuit)$,
\end{enumerate}
then each Lie derivation of $\mathcal {G}$ is of the standard form
$(\spadesuit)$.
\end{corollary}

\begin{proof}
Since the bilinear pairings are both zero, the assumption (1)
implies the condition (1) of Theorem \ref{3.3} and the assumption
(2) implies the condition (2) of Theorem \ref{3.3}.
\end{proof}

\begin{corollary}\cite[Theorem 11]{Cheung}
Every Lie derivation of a triangular algebra $Tri(A,M,B)$ is of the
standard form if the following conditions hold:
\begin{enumerate}
\item[(1)] $W(A)=A$ and every Lie derivation of $A$ is of the
standard form; \item[(2)] $W(B)=B$ and every Lie derivation of $B$
is of the standard form,
\end{enumerate}
\end{corollary}

At the end of this section, let us illustrate an application of
Corollary \ref{3.9}. We will construct a class of algebras with
bilinear pairings being both zero, which is called
\textit{generalized one-point extension algebras} in our situation.
Note that they are not triangular algebras. We observe that each Lie
derivation of a generalized one-point extension algebra is of the
standard form $(\spadesuit)$. Moreover, the standard decomposition
is unique.

\begin{definition}
Let $(\Gamma_0, \Gamma_1)$ be a finite quiver without oriented
cycles and $|\Gamma_0|\geq 2$. Let $\Gamma^{\ast}$ be a quiver
whose vertex set is $\Gamma_0$ and
$$\Gamma_1^{\ast}=\{\alpha^{\ast}: i\rightarrow j\mid \alpha: j\rightarrow i
\,\,\,\text{is an arrow in} \,\,\,\Gamma_1\}.$$ For a path
$p=\alpha_n\cdots\alpha_1$ in $\Gamma$, write the path
$\alpha_1^{\ast}\cdots\alpha_n^{\ast}$ in $\Gamma^{\ast}$ by
$p^{\ast}$. Given a set $\rho$ of relations, denote by
$\Lambda=K(\Gamma, \rho)$. Define the generalized one-point
extension algebra $E(\Lambda)$ to be the path algebra of the quiver
$(\Gamma_0, \Gamma_1\cup\Gamma_1^{\ast})$ with relations
$$\rho\,\,\cup\,\,\rho^{\ast}\,\,
\cup\,\,\{\alpha\beta^{\ast}\mid\alpha,\beta\in\Gamma_1\}\,\,
\cup\,\,\{\alpha^{\ast}\beta\mid\alpha,\beta\in\Gamma_1\}.$$
\end{definition}

In order to study Lie derivations of $E(\Lambda)$, we need the
following lemmas.

\begin{lemma}\label{3.12}
$W(E(\Lambda))=E(\Lambda)$.
\end{lemma}
\begin{proof}
According to the definition, $W(E(\Lambda))$ contains all
idempotents $e_i$. Furthermore, for arbitrary arrow $\alpha$ with
$e(\alpha)=j$, the fact $\alpha=[e_j, \alpha]$ implies that
$\alpha\in W(E(\Lambda))$. With the same reason, $\alpha^{\ast}\in
W(E(\Lambda))$ and then $W(E(\Lambda))=E(\Lambda)$.
\end{proof}

Since $\Gamma$ is a quiver without oriented cycles, we can take a
source $i$ in $\Gamma$. Let $e_i$ be the corresponding idempotent in
$E(\Lambda)$. Then $E(\Lambda)$ is isomorphic to a generalized
matrix algebra $\mathcal{G}=\left[
\begin{array}
[c]{cc}%
A & M\\
N & B\\
\end{array}
\right]$ with $A\simeq E(\Lambda')$, where the quiver $(\Gamma',
\rho')$ of $\Lambda'$ is obtained via removing the vertex $i$ and
the relations starting at $i$. Moreover, in view of the construction
of $E(\Lambda)$ we know that the bilinear pairings are both zero.

\begin{lemma}\label{3.13}
An additive mapping $\Theta_{\rm d}$ is a derivation of $E(\Lambda)$
if and only if $\Theta_d$ has the form
$$
\begin{aligned}
& \Theta_{\rm d}\left(\left[
\begin{array}
[c]{cc}%
a & m\\
n & b\\
\end{array}
\right]\right) =& \left[
\begin{array}
[c]{cc}%
\delta_1(a) & am_0-m_0b+\tau_2(m)\\
n_0a-bn_0+\nu_3(n) & 0\\
\end{array}
\right] ,\\ &  \forall \left[
\begin{array}
[c]{cc}%
a & m\\
n & b\\
\end{array}
\right]\in \mathcal{G},
\end{aligned}
$$
where $m_0\in M, n_0\in N$ and
$$
\begin{aligned} \delta_1:& A \longrightarrow A, &
 \tau_2: & M\longrightarrow M, & \nu_3: & N\longrightarrow N
\end{aligned}
$$
are all $\mathcal{R}$-linear mappings satisfying the following
conditions:
\begin{enumerate}
\item[{\rm(1)}] $\delta_1$ is a derivation of $A$;

\item[{\rm(2)}] $\tau_2(am)=a\tau_{2}(m)+\delta_1(a)m$ and
$\tau_2(mb)=\tau_2(m)b;$

\item[{\rm(3)}] $\nu_3(na)=\nu_3(n)a+n\delta_1(a)$ and
$\nu_3(bn)=b\nu_3(n).$
\end{enumerate}
\end{lemma}

\begin{proof}
Since the bilinear pairings are both zero, by Lemma \ref{3.2}, we
only need to show that $\mu_4=0$. But this is clear by condition (2)
of Lemma \ref{3.2}.
\end{proof}

\begin{lemma}\label{3.14}
Every derivation $\Theta$ of $E(\Lambda)$ with ${\rm
Im}(\Theta)\in \mathcal {Z}(E(\Lambda))$ is zero.
\end{lemma}

\begin{proof}
If there exists a nonzero derivation $\Theta$ of $E(\Lambda)$ with
${\rm Im}(\Theta)\in \mathcal {Z}(E(\Lambda))$, then it follows from
Lemma \ref{3.13} that $\delta_1\neq 0$ and ${\rm Im}(\delta_1)\in
\mathcal {Z}(A)$. Repeating this process continuously, we eventually
get that there exists a nonzero derivation $f$ of $K$. However, this
is impossible.
\end{proof}

\begin{proposition}\label{3.15}
Each Lie derivation of $E(\Lambda)$ is of the standard form
$(\spadesuit)$. Moreover, the standard decomposition is unique.
\end{proposition}

\begin{proof}
Obviously, $W(K)=K$. If $|\Gamma'_0|=1$, then $A\simeq K$ and hence
$W(A)=A$. If $|\Gamma'_0|>1$, then $A$ can be viewed as a
generalized one-point extension too. Thus $W(A)=A$. By Corollary
\ref{3.9}, if each Lie derivation of $A$ has the standard form
$(\spadesuit)$, then so is $E(\Lambda)$. Note that $\Gamma$ is a
finite quiver. Repeat the above process with finite times, we arrive
at the algebra $K$. Of course, each Lie derivation of $K$ is of the
standard form. This implies that each Lie derivation of $E(\Lambda)$
is standard. The uniqueness of the standard decomposition is due to
Corollary \ref{3.14}.
\end{proof}

\section{Lie derivations of dual extensions}

\begin{lemma}\label{4.1}
Let $\Gamma$ be a finite quiver without oriented cycles,
$A=K(\Gamma, \rho)$ and $\mathscr{D}(A)$ be the dual extension of
$A$. Then $\mathscr{D}(A)=W(\mathscr{D}(A))$.
\end{lemma}

\begin{proof}
If $\Gamma$ only contains one vertex, then the algebra
$\mathscr{D}(A)$ is trivial. That is, $\mathscr{D}(A)\simeq K$. In
this case $W(\mathscr{D}(A))=\mathscr{D}(A)$.

Now suppose that the number of vertices in $\Gamma$ is not less than
2. It follows from the condition (3) in the definition of
$W(\mathscr{D}(A))$ that all the trivial paths are contained in
$W(\mathscr{D}\Lambda)$. On the other hand, $\Gamma$ is a quiver
without oriented cycles. Then for an arbitrary arrow
$\alpha\in\Gamma$, we have $\alpha=[\alpha, s(\alpha)]$, which is
due to the fact $\alpha s(\alpha)=\alpha$ and $s(\alpha)\alpha=0$.
The condition (1) of the definition of $W(\mathscr{D}(A))$ shows
that $\alpha\in W(\mathscr{D}(A))$. Similarly, it can be proved that
$\alpha^{\ast}\in W(\mathscr{D}(A))$. Therefore all paths are
contained in $W(\mathscr{D}(A))$. Therefore
$\mathscr{D}(A)=W(\mathscr{D}(A))$.
\end{proof}

\begin{lemma}\label{4.2}
Let $\mathscr{D}(\Lambda)$ be the dual extension of a path algebra
$\Lambda=K(\Gamma, \rho)$. Let $i$ be a source in $\Gamma$ and
$$\mathcal {P}_i=\{p\in \mathscr{P}|s(p)=e(p)=i, p^2=0\}.$$ Denote the vector space
spanned by paths of $\mathcal {P}_i$ by $V$. Assume that
$\Theta_{\rm Lied}$ is a Lie derivation on $\mathscr{D}(\Lambda)$.
Then $\Theta_{\rm Lied}(v)$ is in the center of
$\mathscr{D}(\Lambda)$ for all $v\in V$.
\end{lemma}

\begin{proof} It is easy to see that
$$
\mathscr{D}(\Lambda)\simeq \left[
\begin{array}
[c]{cc}%
(1-e_i)\mathscr{D}(\Lambda)(1-
e_i) & (1- e_i)\mathscr{D}(\Lambda) e_i\\
e_i\mathscr{D}(\Lambda)(1-e_i) & e_i\mathscr{D}(\Lambda) e_i
\end{array}
\right].
$$
Then $\Theta_{\rm Lied}$ has the form described in Lemma \ref{3.1}.
The condition (1) of Lemma \ref{3.1} implies that
$\delta_4(p^{\ast}p)=0$. Thus $\Theta_{\rm
Lied}(p^{\ast}p)=\mu_4(p^{\ast}p)$. It follows from condition (2) of
Lemma \ref{3.1} that $\mu_4(p^{\ast}p)=p^{\ast}q+q'p$. This shows
that $0m=m\mu_4(p^{\ast}p)=0$ and $n0=\mu_4(p^{\ast}p)n=0$ for all
$m\in M$ and $n\in N$. Note that $B$ is commutative. Then
$\Theta_{\rm Lied}(p^{\ast}p)\in \mathcal
{Z}(\mathscr{D}(\Lambda))$.
\end{proof}

Now we are in position to prove the main result of this paper.

\begin{theorem}\label{4.3}
Let $\Gamma$ be a quiver without oriented cycles, $\Lambda=K(\Gamma,
\rho)$ and $\mathscr{D}(\Lambda)$ be the dual extension of
$\Lambda$. Then each Lie derivation on $\mathscr{D}(\Lambda)$ is of
the standard form $(\spadesuit)$.
\end{theorem}

\begin{proof}
 Let $\Theta_{\rm Lied}$ be a Lie
derivation of $\mathscr{D}(\Lambda)$. Suppose that
$\mathscr{D}(\Lambda)$ as vector space has the decomposition
$\mathscr{D}(\Lambda)=V\oplus W$. Define a linear mapping $\Theta'$
by $\Theta'(V)=0$, $\Theta'(W)=\Theta_{\rm Lied}(W)$  and define a
linear mapping $\Delta'$ by $\Delta'(V)=\Theta_{\rm Lied}(V)$ and
$\Delta'(W)=0$. Clearly, $\Theta_{\rm Lied}=\Theta'+\Delta'$.
Furthermore, it follows from Lemma \ref{4.2} that ${\rm
Im}(\Delta')\subset \mathcal {Z}(\mathscr{D}(\Lambda))$. This shows
that $\Theta'$ is also a Lie derivations on $\mathscr{D}(\Lambda)$.
Clearly, if each Lie derivations of type $\Theta'$ is of the
standard form $(\spadesuit)$, then every Lie derivations of
$\mathscr{D}(\Lambda)$ has the standard form $(\spadesuit)$.

Assume that $i$ is a source in $\Gamma$ and $e_i$ the corresponding
idempotent in $\mathscr{D}(\Lambda)$. Let $\Theta_{\rm Lied}$ be a
Lie derivation of $\mathscr{D}(\Lambda)$ satisfying the condition
$\Theta_{\rm Lied}(V)=0$. Then
$$
\mathscr{D}(\Lambda)\simeq \left[
\begin{array}
[c]{cc}%
(1-e_i)\mathscr{D}(\Lambda)(1-
e_i) & (1- e_i)\mathscr{D}(\Lambda) e_i\\
e_i\mathscr{D}(\Lambda)(1-e_i) & e_i\mathscr{D}(\Lambda) e_i
\end{array}
\right].
$$

Let us first prove that for $\mathscr{D}(\Lambda)$, the condition
(2) of Theorem \ref{3.3} is satisfied. We have from the construction
of $\mathscr{D}(\Lambda)$ that $ e_i\mathscr{D}(\Lambda) e_i$ is an
algebra with a basis $\{p^{\ast}p\mid s(p)=i\}.$ Furthermore, if $p,
q$ are nontrivial, then $(p^{\ast}p)(q^{\ast}q)=0$. Thus the algebra
$ e_i\mathscr{D}(\Lambda) e_i$ is commutative. Let $l_B$ be equal to
$\mu_4$. Then $l_B([b, b'])=0$ for all $b, b'\in
e_i\mathscr{D}(\Lambda) e_i$. Note that $(1-
e_i)\mathscr{D}(\Lambda) e_i \mathscr{D}(\Lambda)(1- e_i)=0$. That
is, $\Phi_{MN}=0$. We conclude that $\mu_1(mn)=0$ for all $m\in M$
and $n\in N$. On the other hand, since $\Theta( p^{\ast} p)=0$ for
all nontrivial path $ p$, we arrive at $\mu_4( p^{\ast} p)=0$ and
hence $\mu_4(nm)=0$. Therefore $l_B(nm)=\mu_1(mn)$.

Let $b=ke_i+v$, where $v\in V$. It follows from the structure of
$M$ and $\mathscr{D}(\Lambda)$ that
$$\tau_2(mb)=k\tau_2(me_i)=k\tau_2(m)=\tau_2(m)ke_i=\tau_2(m)(ke_i+v)=\tau_2(m)b.$$
Similarly, we can obtain that $\nu_3(bn)=b\nu_3(n)$. Then it follows
from conditions (4) and (5) of Lemma \ref{3.1} that
$l_B(b)n=n\delta_4(b)$ and $ml_B(b)=\delta_4(b)m$.

By the definition of dual extension, it is easy to check that
$$
(1- e_i)\mathscr{D}(\Lambda)(1-
e_i)\simeq\mathscr{D}(\Lambda'),
$$ where $\Lambda'= K(\Gamma',
\rho')$, $(\Gamma', \rho')$ is a quiver obtained by removing the
vertex $i$ and the relations starting at $i$. Clearly, $\Gamma'$ has
no oriented cycles. Then Lemma \ref{4.1} implies that
$\mathscr{D}(\Lambda')=W(\mathscr{D}(\Lambda'))$. Thus $\Theta_{\rm
Lied}$ is of the standard form $(\spadesuit)$ if each Lie derivation
on $\mathscr{D}(\Lambda')$ is standard.

Note that $\Gamma$ is a finite quiver. Repeating this process with
finite times, we arrive at the algebra $K$. That is, if each Lie
derivation of $K$ is standard, then so is $\mathscr{D}(\Lambda)$.
This completes the proof.
\end{proof}

\begin{corollary}\label{4.4}
Let $\Theta_{\rm Lied}$ be a Lie derivation of
$\mathscr{D}(\Lambda)$. Then there exists a derivation $D$ of
$\mathscr{D}(\Lambda)$ such that $\Theta_{\rm Lied}(x)=D(x)$ for
all $x=\sum_ik_i p_i\in \Lambda$, where $p_i$ are non-trivial
paths.
\end{corollary}

\begin{proof}
Let $\Theta_{\rm Lied}$ be a Lie derivation of
$\mathscr{D}(\Lambda)$. We have from Theorem \ref{4.3} that
$\Theta_{\rm Lied}$ is of the standard form $(\spadesuit)$. Then
$\Theta_{\rm Lied}=D+\Delta$, where $D$ is a derivation of
$\mathscr{D}(\Lambda)$ and $\Delta(x)\in  \mathcal
{Z}(\mathscr{D}(\Lambda))$ for all $x\in \mathscr{D}(\Lambda)$. Note
that $\Delta$ is also a Lie derivation of $\mathscr{D}(\Lambda)$.
Then for a path $p$ with $s(p)\neq e(p)$, the fact $p=[p, s(p)]$
gives
$$
\Delta(p)=[\Delta(p), s(p)]+[p, \Delta(s(p))].
$$ It
follows from the image of $\Delta$ being in $ \mathcal
{Z}(\mathscr{D}(\Lambda))$ that $\Delta(p)=0$. Moreover, let $p$ be
a nontrivial path with $s(p)=e(p)$. By the construction of
$\mathscr{D}(\Lambda)$, $p$ is of the form $x^{\ast}x$, where $x$ is
a nontrivial path in $\Gamma$. Therefore
$$
\Delta(p)=\Delta(x^{\ast}x)=\Delta([x^{\ast},
x])=[\Delta(x^{\ast}), x]+[x^{\ast}, \Delta(x)]=0.
$$
Then for all $x=\sum_ik_i\overline p_i\in \Lambda$, where $p_i$ are
non-trivial paths, we have $\Theta_{\rm Lied}(x)=D(x).$
\end{proof}

Let us address problem of whether the standard decomposition of
arbitrary Lie derivation of $\mathscr{D}(\Lambda)$ is unique.
Fortunately, the answer is positive. In order to give the answer, we
are forced to characterize the center of $\mathscr{D}(\Lambda)$.

\begin{lemma}\label{4.5}
Let $\Gamma$ be a connected quiver with $|\Gamma_0|\geq 2$. Then
the elements in $\mathcal {Z}(\mathscr{D}(\Lambda))$ are of the
form
$$
k+\sum_{e(p)=s(p), p^2=0} k_pp.
$$
\end{lemma}

\begin{proof}
Assume that
$$
x=\sum_{i\in \Gamma_0}k_ie_i+\sum_{s(p)\neq
e(p)}k_pp+\sum_{s(p)=e(p), p^2=0}k_pp\in\mathcal
{Z}(\mathscr{D}(\Lambda)).
$$
Applying the fact $e_tx=xe_t$ yields that
$$
\sum_{t=s(p)\neq e(p)}k_pp=\sum_{t=e(p)\neq s(p)}k_pp.
$$
This implies that for all paths $p$ with $s(p)\neq e(p)$, we have
$k_p=0$ if $s(p)=t$ or $e(p)=t$. Since $t$ is arbitrary, the
coefficients of all paths $p$ with $s(p)\neq e(p)$ are zero.

Let $\alpha$ be an arrow in $\Gamma_1$ with $e(\alpha)=j$ and
$s(\alpha)=t$. In view of $\alpha x=x\alpha$, we know that
$k_j=k_t$. Note that $\Gamma$ is a connected quiver. Thus $k_i=k$
for all $i\in \Gamma_0$, where $k\in K$. This completes the proof.
\end{proof}

\begin{lemma}\label{4.6}
Let $D$ be a derivation of $\mathscr{D}(\Lambda)$. If ${\rm
Im}(D)\subset \mathcal {Z}(\mathscr{D}(\Lambda))$, then $D=0$.
\end{lemma}

\begin{proof}
Let $D$ be a derivation of $\mathscr{D}(\Lambda)$ with ${\rm
Im}(D)\subset \mathcal {Z}(\mathscr{D}(\Lambda))$. Clearly, $D$ is
also a Lie derivation of $\mathscr{D}(\Lambda)$. By the proof of
Corollary \ref{4.4}, we have $D(p)=0$ for all nontrivial path $p$.
We now prove $D(e_i)=0$ for all $i\in \Gamma_0$. According to
Corollary \ref{4.5}, we can assume that $D(e_i)=k_i+\sum_{e(p)=s(p),
p^2=0} k_p^ip$. Note that $e_i$ is an idempotent. By the definition
of derivation, it is easy to verify that $k_i=0$ and $k_p^i=0$ for
paths $p$ with $s(p)=i$. Furthermore, suppose there exists some $p$
with nonzero coefficient in $D(e_i)$. Let $s(p)=j\neq i$. Then
$D(e_ie_j)=0$. On the other hand, $D(e_ie_j)=D(e_i)e_j+e_iD(e_j)\neq
0$, which is a contradiction. This implies that $D=0$.
\end{proof}

As a direct consequence of Lemma \ref{4.6} we immediately get

\begin{proposition}
Let $\Theta_{\rm Lied}$ be a Lie derivation of a dual extension
algebra $\mathscr{D}(\Lambda)$. Then the standard decomposition of
$\Theta_{\rm Lied}$ is unique.
\end{proposition}

\begin{remark}
On one hand, we know that a Lie derivation of dual extension algebra
can be uniquely expressed as the sum of a derivation and a linear
mapping annihilating all commutators with images in the center of
the algebra. On the other hand, the sum of a derivation and a linear
mapping annihilating all commutators with images in the center is
clearly a Lie derivation. In this sense, the forms of Lie derivation
on dual extensions are thoroughly characterized at all.
\end{remark}

Now let us give an example of a Lie derivation which is not a
derivation.

\begin{example}
Let $\Gamma$ be the following quiver
$$
\xymatrix@C=25mm{
\bullet \ar@<0pt>[r]^(0){1}^{\alpha}  &
\bullet
  \ar@<0pt>[r]^(0.3){\beta}^(0){2}^(0.7){3}&\bullet}
$$
with relation $\beta\alpha$ and $\Lambda=K(\Gamma, \rho)$. Let
$\mathscr{D}(\Lambda)$ be the dual extension of the path algebra
$\Lambda$. Define a linear mapping $\Theta_{\rm Lied}$ on
$\mathscr{D}(\Lambda)$ by
\begin{eqnarray*}
  \Theta_{\rm Lied}(e_1)=k_1+\alpha^{\ast}\alpha; & \Theta_{\rm Lied}(e_2)=k_2+\beta^{\ast}\beta; & \Theta_{\rm Lied}(e_3)=k_3;\\
  \Theta_{\rm Lied}(\alpha)=\alpha;  &
  \Theta_{\rm Lied}(\alpha^{\ast})=\alpha^{\ast}; &
  \Theta_{\rm Lied}(\beta)=\beta^{\ast}; \\
  \Theta_{\rm Lied}(\beta^{\ast})=\beta^{\ast}; &
  \Theta_{\rm Lied}(\alpha^{\ast}\alpha)=2\alpha^{\ast}\alpha;&
  \Theta_{\rm Lied}(\beta^{\ast}\beta)=2\beta^{\ast}\beta
\end{eqnarray*}
Then a direct computation shows that $\Theta_{\rm Lied}$ is a Lie
derivation of  $\mathscr{D}(\Lambda)$ but not a derivation.

Moreover, we give the standard decompositions of $\Theta_{\rm Lied}$
here. Define a linear mapping $\Delta$ on $\mathscr{D}(\Lambda)$ by
\begin{eqnarray*}
\Delta(e_1)=k_1+\alpha^{\ast}\alpha, &
\Delta(e_2)=k_2+\beta^{\ast}\beta, &  \Delta(e_3)=k_3
\end{eqnarray*}
and let $D=\Theta_{\rm Lied}-\Delta$, then $\Theta_{\rm
Lied}=D+\Delta$ is the standard decomposition of $\Theta$.
\end{example}

\bigskip

\noindent{\bf Acknowledgement} The first author would like to
express his sincere thanks to Chern Institute of Mathematics of
Nankai University for the hospitality during his visit. He also
acknowledges Prof. Chengming Bai for his kind consideration and warm
help.

\end{document}